\documentclass[12pt]{amsart}
\usepackage[margin=1.in, centering]{geometry}
\usepackage{bm}
\usepackage{csquotes}
\usepackage{mathtools}
\usepackage{amssymb,latexsym,amsmath,extarrows, mathrsfs, amsthm,cite}
\usepackage{graphicx}
\usepackage{tikz}
\usepackage{amsmath}
\usepackage{chngcntr}
\counterwithin{equation}{section}
\usepackage[colorlinks, linkcolor=blue]{hyperref}
\usepackage{bbm}
\usepackage{fancyhdr}
\usepackage{color}
\usepackage{wasysym}
\usepackage{float}

\newtheorem{theorem}{Theorem}[section]
\newtheorem{lemma}[theorem]{Lemma}
\newtheorem{proposition}[theorem]{Proposition}

\theoremstyle{remark}

\allowdisplaybreaks

\title{On the Hyperbolic Fractional Sum of the Divisor Function
}

\author{Ling Li  }
\address{School of Mathematics, Jilin University, Changchun, 130012, People's Republic of China}
\email{liling24@mails.jlu.edu.cn}


\begin{document}

\begin{abstract}
Let $\tau(n)$ denote the classical divisor function. In this paper, we consider the hyperbolic fractional sum of the divisor function defined by
$$T(x) = \sum_{n_1 n_2 \leqslant x} \tau\left( \left[ \frac{x}{n_1 n_2} \right] \right) = \sum_{n \leqslant x} \tau\left( \left[ \frac{x}{n} \right] \right) \tau(n),
$$
where $[t]$ denotes the integral part of the real number $t$. By establishing new estimates for a class of three-dimensional exponential sums with  constant perturbation, we obtain an improved asymptotic formula for $T(x)$. In particular, we show that for any $\varepsilon > 0$, the error term in the asymptotic expansion of $T(x)$ is bounded by $O(x^{17/30+\varepsilon})$. This result breaks the $4/7$-barrier which corresponds to the application of the classical divisor problem conjecture $1/4+\varepsilon$.

\end{abstract}

\maketitle
\section{Introduction}

Let $\tau(n)$ denote the divisor function,  the number of positive divisors of an  integer $n\geqslant1$. A classical result in analytic number theory, established by Dirichlet, provides the following asymptotic formula
\begin{eqnarray}\label{sum tau(x)}
    \sum_{n\leqslant x}\tau(n)=x(\log{x}+2\gamma-1)+\Delta(x),
\end{eqnarray}
where $\gamma$ is the Euler-Mascheroni constant and $\Delta(x)$ denotes the error term.
While Dirichlet initially showed that $\Delta(x) \ll x^{1/2}$, the problem of determining the optimal upper bound for $\Delta(x)$ has a long history. Following contributions from various authors \cite{Huxley, Iwanice, Kolesnik, Var, Var2, Vor}, the most recent result is due to Bourgain and Watt \cite{BW}, who proved that  $\Delta(x)\ll x^{517/1648+\varepsilon}$ for any $\varepsilon > 0$.
A classical conjecture states that $$\Delta(x) \ll x^{\frac{1}{4}+\varepsilon},$$ although this remains an open problem.

Recalling the definition of $\tau(n)$, it is straightforward to observe that
$$
\sum_{d\leqslant x}\left[\frac{x}{d}\right]=\sum_{d\leqslant n}\sum_{\substack{n\leqslant x\\ d|n}}1=\sum_{n\leqslant x}\sum_{d|n}1=\sum_{n\leqslant x}\tau(n),$$
where $[t]$ is the integral part of the real number $t$.
Consequently, we may also view \eqref{sum tau(x)} as an asymptotic formula for $\sum_{n\leqslant x}[x/n]$.
Motivated by this connection, Bordell\`{e}s, Dai, Heyman,
 Pan, and Shparlinski \cite{BDHPS} generalized the investigation to consider the class of fractional sums defined by \begin{eqnarray*}\label{S_g(x)}
S_f(x):=\sum_{n\leqslant  x}f\left(\left[\frac{x}{n}\right]\right),
 \end{eqnarray*}
where $f$ is an arithmetic function. Specifically, for any arithmetic function $f$ satisfying the  condition $|f| \leqslant A\tau_k$, they proved the following asymptotic formula
$$
\sum_{n \leqslant x} f\left(\left[\frac{x}{n}\right]\right) = x \sum_{n=1}^{\infty} \frac{f(n)}{n(n+1)} + O \left( A x^{\frac{1}{2}} (\log x)^{\delta_k + \frac{\varepsilon_{k+1}(x)}{2}} \right),
$$
where $A>0$ and $\tau_k(n)$ is the $k$-fold divisor function,
\(\delta_1 = 0\) for \(k = 1\),  $\delta_k = k - 1/2$  for \(k \geqslant 2\), 
$$
\varepsilon_0(x)=1,\quad \varepsilon_k(x) = \sqrt{\frac{k \log \log \log x}{\log \log x}} \left( k - 1 + \frac{30}{\log \log \log x} \right).
$$
For a more comprehensive review of the research on $S_f(x)$, we refer the reader to \cite{B, LiMa, LM2, LWY, LWY2, MS, MS2, MW, MWZ, S, W2, W, Z2, Z, ZWG, ZWG2, ZWG3, ZW} 
and the references therein.
More recently, Karras, Li and Stucky \cite{KLS}
extended the scope of such investigations by introducing the ``hyperbolic fractional sum", which is  defined as
\begin{eqnarray*}
    T_f(x):=\sum_{n_1n_2\leqslant x}f\left(\left[\frac{x}{n_1n_2}\right]\right)=\sum_{n\leqslant x}f\left(\left[\frac{x}{n}\right]\right)\tau(n).
\end{eqnarray*}
Under the assumption of  conjecture $\Delta(x)\ll x^{1/4+\varepsilon}$ and $f(n)\ll n^{\alpha}$ for some  $\alpha\in [0,1)$, it is not hard to obtain that
\begin{eqnarray}\label{T_f(x)}
   T_f(x)=C_1(f)x\log{x}+C_2(f)x+O\left(x^{\frac{4+3\alpha}{7}+\varepsilon}\right),
\end{eqnarray}
where the constants are given by
$$
C_1(f)=\sum_{d=1}^{\infty}\frac{f(d)}{d(d+1)}, \quad C_2(f)=(2\gamma-1)C_1(f)-C_3(f), 
$$
and 
$$
C_3(f)=\sum_{d=1}^{\infty}f(d)\left(\frac{\log{d}}{d}-\frac{\log(d+1)}{d+1}\right).
$$
Remarkably, 
Karras et al.\ \cite{KLS} established the asymptotic formula \eqref{T_f(x)} unconditionally---independent of the conjecture $\Delta(x) \ll x^{1/4}$---by using an estimate for three-dimensional exponential sums due to Robert and Sargos \cite{RS}.

In this paper, we consider the hyperbolic fractional sum $T_f(x)$ in the specific case where $f(n) = \tau(n)$. By  the double large sieve inequality, van der Corput’s B-process, and new estimates for three-dimensional exponential sums with constant perturbations, we are able to improve the error term in \eqref{T_f(x)}. In particular, this result breaks the $4/7$-barrier. Our main theorem is as follows.

\begin{theorem}\label{Main Theorem}
    With notation as above, for any $x\geqslant 1$ and $\varepsilon>0$, we have
    $$
    T(x)=\sum_{n\leqslant x}\tau\left(\left[\frac{x}{n}\right]\right)\tau(n)=
    C_1x\log{x}+C_2x+O\left(x^{\frac{17}{30}+\varepsilon}\right).
    $$
\end{theorem}

The proof of Theorem \ref{Main Theorem}
is established through the study of
 a class of three-dimensional exponential sums with constant perturbations. Such sums generalize the classical estimates for three-dimensional exponential sums  defined by
 $$
S_0(H,M,M):=\sum_{h\sim H}\sum_{m\sim M}\sum_{n\sim N}a(h,m)b(n)e\left(X\frac{h^{\beta}m^{\gamma}n^{\alpha}}{H^{\beta}M^{\gamma}N^{\alpha}}\right),
$$
where $e(t)=e^{2i\pi t}$, $H, M, N \in \mathbb{Z}^+$, $X>1$. The complex numbers $a(h,m)$ and $b(n)$ satisfy $|a(h, m)|, |b(n)| \leqslant 1$, and $\alpha, \beta, \gamma$ are fixed real numbers.
 For $\alpha(\alpha - 1)\beta\gamma \neq 0$, 
 Fouvry and Iwaniec \cite{FI} 
established that
 $$
S_0(H,M,N)\ll (HMN)^{1+\varepsilon}\Bigg\{\bigg(\frac{X}{HMN^2}\bigg)^{\frac{1}{4}}+\frac{1}{N^{\frac{3}{10}}}+\frac{1}{(HM)^{\frac{1}{4}}}+\frac{N^{\frac{1}{10}}}{X^{\frac{1}{4}}}\Bigg\}
$$
 by applying the double large sieve inequality of Bombieri and Iwaniec \cite{BI}.
 Subsequently, this upper bound was improved by Sargos and Wu \cite{SW} for $\alpha(\alpha-1)(\alpha-2)\beta\gamma\neq1$, and by Robert and Sargos \cite{RS} for $\alpha(\alpha-1)\beta\gamma\neq1$, respectively.
More recently, Liu, Wu, and Yang \cite{LWY} generalized $S_0(H,M,N)$ by considering a class of three-dimensional exponential sums with constant perturbation $\delta$
\begin{eqnarray}\label{S-delta}
S_{\delta}(H,M,N):=\sum_{h\sim H}\sum_{m\sim M}\sum_{n\sim N}a(h,m)b(n)e\left(X\frac{M^\beta N^\gamma}{H^\alpha}\frac{h^\alpha}{m^\beta n^\gamma +\delta}\right), 
\end{eqnarray}
where $\alpha>0$, $\beta>0$, $\gamma>0$ and $\delta \in\mathbb{R}$ is a constant. 
They proved that
\begin{eqnarray*}\label{LWY}
    &&S_\delta(H,M,N)\ll \bigg( (XHMN)^{\frac{1}{2}}+(HM)^{\frac{1}{2}}N+HMN^{\frac{1}{2}}+X^{-\frac{1}{2}}HMN\bigg)X^{\varepsilon},\\
     && S_\delta(H,M,N)\ll \bigg((X^\kappa H^{2+\kappa} M^{2+\kappa} N^{1+\kappa+\lambda})^{\frac{1}{2+2\kappa}}+HMN^{\frac{1}{2}}
+(HM)^{\frac{1}{2}}N+X^{-\frac{1}{2}}HMN\bigg)X^{\varepsilon},
\end{eqnarray*}
hold uniformly for $M\geqslant 1$, $N\geqslant1$, $H\leqslant M^{\beta-1}N^{\gamma}$, and $|\delta|\leqslant1/\varepsilon$, where $(\kappa, \lambda)$ is an exponent pair. 
By choosing the exponent pair
$(\kappa, \lambda) = (1/2, 1/2)$, they improved a previous result of Bordell\`{e}s \cite{B}. In particular, they improved the error term in the following asymptotic formula
\begin{eqnarray}\label{LWY-Mangoldt}
S_\Lambda (x)= \sum_{n \leqslant x} \Lambda\left(\left[ \frac{x}{n} \right]\right) = x \sum_{d =1}^\infty \frac{\Lambda(d)}{d(d+1)} + \text{error term}
\end{eqnarray}
from $x^{97/203+\varepsilon}$ to $x^{9/19+\varepsilon}$.
Based on these developments, Li and Ma \cite{LiMa}
investigated the perturbed sum $S_\delta(H,M,N)$ under the condition $X \leqslant (8\delta)^{-1} KM^{\beta}N^{\gamma}$ for $K\geqslant1$.  They proved the following bound
\begin{eqnarray*}\label{LM}
    S_\delta(H,M,N)\ll (HMN)^{1+\varepsilon}\Bigg\{\bigg(\frac{KX}{HMN^2}\bigg)^{\frac{1}{4}}+\bigg(\frac{K^2}{HM}\bigg)^{\frac{1}{4}}+\bigg(\frac{K}{N}\bigg)^{\frac{1}{2}}+\frac{K}{X^{\frac{1}{2}}}\Bigg\},
\end{eqnarray*}
which yielded a further refinement of the error term in the asymptotic formula \eqref{LWY-Mangoldt}.

In this paper, we also consider a class of three-dimensional exponential sums with constant perturbation, similar in structure to those given in \eqref{S-delta}. 
Let 
\begin{eqnarray}\label{type1}
    S(H,M,N):=\sum_{m\sim M}\sum_{n\sim N}a(m,n)\sum_{h\sim H}b(h)e\left(X\frac{(mn+\delta)^{\beta}}{M^\beta N^\beta}\frac{h^\alpha}{H^ \alpha}\right)
\end{eqnarray}
 where $H$, $M$, $N$ are positive integers, $X>1$ is a real number, $a(m,n)$ and $b(h)$ are complex numbers
 such that $|a(m,n)|\leqslant 1$ and $|b(h)|\leqslant 1$,   moreover $\alpha$, $\beta$ are fixed real numbers such that $\alpha(\alpha-1)\beta\neq1$ and $\delta\in \mathbb{R}$ is a constant satisfying $|\delta| \leqslant MN$.
By  the double large sieve inequality,  we obtain the following upper bound for $S(H, M, N)$.
\begin{theorem}\label{S_1(H,M,N)}
    Under the previous notation, for any $\varepsilon>0$, we have
    $$
S(H,M,N)\ll (HMN)^{1+\varepsilon}\left\{\left(\frac{X}{MNH^2}\right)^{\frac{1}{4}}+\frac{1}{(MN)^{\frac{1}{4}}}+\frac{1}{H^{\frac{1}{2}}}+\frac{1}{X^{\frac{1}{2}}}\right\}.
$$
\end{theorem}

Furthermore, let the maximum partial sum of a sequence of complex numbers $\{z_n\}_{n=1}^N$ be defined by 
$$
\left|\sum_{1\leqslant n\leqslant N}z_n\right|^*=\max_{1\leqslant N_1\leqslant N_1\leqslant N}\left|\sum_{n=N_1}^{N_2}z_n\right|.
$$
 Applying Theorem \ref{S_1(H,M,N)}, we derive the following result.

\begin{theorem}\label{S_1^{dag}(H,M,N)}
  Let
  $$
  S^{*}(H,M,N):=\sum_{m\sim M}\sum_{n\sim N}\left|\sum_{h\sim H}e\left(X\frac{(mn+\delta)^\beta }{M^\beta N^\beta }\frac{h^\alpha}{H^\alpha}\right)\right|^{*}.
  $$
With the same notation as in Theorem \ref{S_1(H,M,N)}, we have
$$S^{*}(H,M,N)\ll(HMN)^{1+\varepsilon}\left\{\left(\frac{X}{MNH^2}\right)^{\frac{1}{4}}+\frac{1}{H^{\frac{1}{2}}}+\frac{1}{X}\right\}.$$
\end{theorem}

\noindent{\bf{Notations.}}  
 Let $[x]$ denote the integral part of $x$,  $\{x\} = x - [x]$ represent its fractional part, and  $\|x\|$ to denote the distance of $x$ from the nearest integer. We write $e(x) = e^{2\pi i x}$ for brevity and $d \sim D$ is shorthand for $D < d \leqslant 2D$. The Landau–Vinogradov symbols $\ll, \gg$, and $O$ are used with their standard meanings, and $f(x) \asymp g(x)$ signifies that $f(x) = O(g(x))$ and $g(x) = O(f(x))$ simultaneously. And $\sideset{}{^*}\sum_{a \leqslant n \leqslant b}$ indicates that the terms with $n=a$ or $n=b$ (if they are integers) are to be halved. Throughout this paper, $\varepsilon$ always denotes an arbitrarily small positive constant.

 \section{Preliminaries}
In this section, we provide the preliminary lemmas required to prove Theorems \ref{S_1(H,M,N)} and \ref{S_1^{dag}(H,M,N)}. Following these preparations, we will proceed to the proofs of Theorems \ref{S_1(H,M,N)} and \ref{S_1^{dag}(H,M,N)}, while the more intricate proof of Theorem \ref{Main Theorem} is reserved for the end of the paper.

\subsection{Diophantine inequality}

\leavevmode \\

Here, we record several lemmas concerning Diophantine inequalities that will be instrumental in the proof of Theorem \ref{S_1(H,M,N)}. The first lemma, which originates from the double large sieve inequality, provides a bound for bilinear exponential sums.

\begin{lemma}\cite[Lemma 8]{RS}\label{double large sieve inequality}
   Set
   $$
   S=\sum_{k=1}^{K}\sum_{l=1}^{L}a(k)b(l)e\bigg(Xu(k)v(l)\bigg),
   $$
   where $K$ and $L$ are positive integers; 
   $a(k)$ and $b(l)$ are complex numbers with modulus at most one for $k=1,\ldots,L$; 
   $u(k)$ and $v(l)$ are real number such that $u(k)\ll 1$ and $v(l)\ll 1$, for $k=1,\ldots,K$ and $l=1,\ldots, L$. 
   $X$ is a real number, $X\geqslant 1$. Define $\mathscr{B}_1$ and $\mathscr{B}_2$ as
   \begin{eqnarray*}
   \mathscr{B}_1=\sum_{\substack{1\leqslant k_1,k_2\leqslant K\\|u(k_1)-u(k_2)|\leqslant \frac{1}{X}}}1, \quad \quad \mathscr{B}_2=\sum_{\substack{1\leqslant l_1,l_2\leqslant L\\|u(l_1)-u(l_2)|\leqslant \frac{1}{X}}}1.
   \end{eqnarray*}%
   Then, we  have
   $$
   S\ll X^{\frac{1}{2}}\mathscr{B}_1^{\frac{1}{2}}\mathscr{B}_2^{\frac{1}{2}}.
   $$
\end{lemma}

Next, we provide several lemmas regarding Diophantine inequalities, which will play a crucial role in subsequent proofs.

\begin{lemma}\cite[Theorem 2]{RS}\label{lemma:l1-4}
Let $\beta \in \mathbb{R} \setminus \{0,1\}$ be fixed. For $L\geqslant 2$ and $X>0$. Define
\begin{eqnarray*}
\mathscr{B}_3:=\Bigg|\left\{(l_1,l_2,l_3,l_4)\Big|\;l_i\sim L, i=1,\ldots,4, \bigg|\frac{l_1^{\beta}}{L^{\beta}}+\frac{l_2^{\beta}}{L^{\beta}}-\frac{l_3^{\beta}}{L^{\beta}}-\frac{l_4^{\beta}}{L^{\beta}}\bigg|\leqslant \frac{1}{X}\right\}\Bigg|.
\end{eqnarray*}%
Then $$\mathscr{B}_3\ll L^{4+\varepsilon}\left(\frac{1}{L^2}+\frac{1}{X}\right).$$
\end{lemma}

\begin{lemma}\label{lemma:l1-2}
    Let $\beta\in \mathbb{R}\setminus \{0\}$ be fixed. For $L, X\geqslant1$, define 
    $$
    \mathscr{B}_4:=\Bigg|\left\{(l_1,l_2)\Big|\;l_1,l_2\sim L,\Big|\frac{L^\beta}{l_1^\beta}-\frac{L^\beta}{l_2^\beta}\Big|\leqslant \frac{1}{X}\right\}\Bigg|.
    $$
    Then we have
    $$
    \mathscr{B}_4\ll L^2\left(\frac{1}{L}+\frac{1}{X}\right).
    $$
\end{lemma}

\begin{proof}
    Define
    $$
    f(t)=\left(\frac{L}{t}\right)^{\beta}.
    $$
    By the Mean Value Theorem, for any $l_1, l_2 \in (L, 2L]$, there exists $\xi$ between $l_1$ and $l_2$ such that $$\left|f(l_1) - f(l_2)\right| = \left|f^{'}(\xi)\right| \left|l_1 - l_2\right|.$$
    A simple calculation shows that 
    $$
    \left|f^{'}(\xi)\right|\asymp L^{-1}
    $$
    for $\xi\in(L,2L]$.
    Specifically, we have
    $$
      C_1\frac{|l_1-l_2|}{L}\leqslant \left|\left(\frac{L}{l_1}\right)^{\beta}-\left(\frac{L}{l_2}\right)^{\beta}\right|\leqslant C_2 \frac{|l_1-l_2|}{L}.
    $$
    Moreover, as $|f(l_1)-f(l_2)|\leqslant \frac{1}{X}$, we obtain
    $$
    C_1\frac{|l_1-l_2|}{L}\leqslant \frac{1}{X}\quad \Rightarrow \quad \left|l_1-l_2\right|\leqslant \frac{L}{C_1X}.
    $$
    For a fixed $l_1$, any qualifying $l_2$ must fall within the interval $[l_1 - \frac{L}{C_1 X}, l_1 + \frac{L}{C_1 X}]$. The length of this interval is $\frac{2L}{C_1 X}$, and the number of integers contained in it is $O\left(1 + \frac{L}{X}\right)$.
    Thus, we have
    $$
    \mathscr{B}_4\ll L\left(1+\frac{L}{X}\right).
    $$
\end{proof}

\begin{lemma}\label{lemma：mn}
    Let $\beta\in \mathbb{R}\setminus \{0\}$ and $\delta\in \mathbb{R}$ be constants satisfying $|\delta| \leqslant MN$. For $M, N\geqslant 1$ and $X>1$, we define
    $$
    \mathscr{B}_5:=\Bigg|\left\{(m_1,m_2,n_1,n_2)\Big|m_1,m_2\sim M, n_1,n_2\sim N, \bigg|\frac{(m_1n_1+\delta)^{\beta}}{M^\beta N^\beta}-\frac{(m_2n_2+\delta)^{\beta}}{M^\beta N^\beta}\bigg|\leqslant\frac{1}{X}\right\}\Bigg|.
    $$
    Then
    $$
    \mathscr{B}_5\ll (MN)^{2+\varepsilon}\left(\frac{1}{MN}+\frac{1}{X}\right).
    $$
\end{lemma}

\begin{proof}
We partition the elements of the set according to the products $h_i = m_i n_i + \delta$. Let $\mathcal{H}$ denote the set of possible values for $h_i$, so that $|\mathcal{H}| =H\sim  MN$. For fixed $h_1$ and $h_2$, the number of pairs $(m_i, n_i)$ satisfying $m_i n_i = h_i - \delta$ is given by the divisor function $\tau(h_i - \delta)$.
Using the well-known bound $\tau(n) \ll n^\varepsilon$, we have
    \begin{eqnarray*}
        \mathscr{B}_5
        &=&\sum_{\substack{h_1,h_2\sim H\\
        \left|\frac{h_1^{\beta}}{M^\beta N^\beta}-\frac{h_2^\beta}{M^\beta N^\beta}\right|\leqslant\frac{1}{X}}}\left(\sum_{m_1n_1=h_1-\delta}1\right)\left(\sum_{m_2n_2=h_2-\delta}1\right)\\
        &\ll&(MN)^{2\varepsilon }\sum_{\substack{h_1,h_2\sim H\\
        \left|\frac{h_1^{\beta}}{M^\beta N^\beta}-\frac{h_2^\beta}{M^\beta N^\beta}\right|\leqslant\frac{1}{X}}}1\\
        &\ll&(MN)^{2+\varepsilon }\left(\frac{1}{MN}+\frac{1}{X}\right).
    \end{eqnarray*}
    In the last inequality, we have used Lemma \ref{lemma:l1-2}.
\end{proof}

\subsection{Basic lemmas on exponential sums}

\leavevmode \\

In this subsection, we present several lemmas for estimating exponential sums, which will be instrumental in proving Theorem \ref{S_1^{dag}(H,M,N)}. We begin with the Kusmin-Landau inequality.

\begin{lemma}\cite[Theorem 2.1]{GK}\label{Kusmin-Landau}
   If $f$ is continuously differentiable, $f^{'}$ is monotonic, and $\parallel f^{'}\parallel\geqslant \lambda>0$ on $I$ then
   $$\sum_{n\in I}e\left(f(n)\right)\ll \lambda^{-1}.$$
\end{lemma}

The second lemma provides a connection between the maximum of partial sums and an integral of exponential sums, often used in the study of maximal inequalities.

\begin{lemma}\cite[Lemma 2]{RS}\label{Z*}
   Let $z_1,\ldots, z_N$ be complex numbers. We set
  $$\left|\sum_{1\leqslant n\leqslant N}z_n\right|^{*}=\max_{1\leqslant N_1\leqslant N_2\leqslant N}\left|\sum_{n=N_1}^{N_1}z_n\right|,
  $$
  letting $k\geqslant1$ be a real number. We have
  $$\left(\left|\sum_{1\leqslant n\leqslant N}z_n\right|^{*}\right)^{k}\leqslant \left(1+\log N\right)^{k-1}\int_{-\frac{1}{2}}^{\frac{1}{2}}\left|\sum_{n=1}^{N}z_ne(nt)\right|^{k}\mathscr{L}(t)\mathrm{d}t,
  $$
  with $\mathscr{L}=\min\left\{N,\frac{1}{2|t|}\right\}$ 
  and $\int_{-\frac{1}{2}}^{\frac{1}{2}}\mathscr{L}(t)\mathrm{d}t=1+\log N.$
\end{lemma}

After the application of the standard $B$-process, the endpoints of the summation interval $[a, b]$ for the exponential sum depend on the parameters $x$ and $M_1$, which may complicate certain composite estimates. The following lemma is designed to bypass this dependency.

\begin{lemma}\cite[Lemma 4]{RS}\label{dependence}
   There exist two integers $L$ and $L_1$ which do not depend on $ x\in[X,2X]$ nor $M_1\in \{M+1,\ldots, 2M\}$ and a real number $\rho=\rho(u),$ with
   $$
   L\asymp L_1 \asymp\frac{X}{M},\quad 1\leqslant L\leqslant L_1\ll L, \quad \rho \asymp X
   $$
   such that we have
   $$
   \sum_{m=M+1}^{M_1}e\left(x\left(\frac{m}{M}\right)^{\alpha}\right)\ll \frac{M}{X^{\frac{1}{2}}}\left|\sum_{L\leqslant l\leqslant L_1}e\left(\rho(x)\left(\frac{l}{L}\right)^{\overline{\alpha}}\right)\right|^{*} +M^{\frac{1}{2}}
   $$
   with $\overline{\alpha}=\frac{\alpha}{\alpha-1}$. We may take $L=\left[2^{-|\alpha|-1}|\alpha|XM^{-1}\right]$ 
   and $L_1=\left[2^{|\alpha|+2}|\alpha|XM^{-1}\right]$
   and $\rho(x)=(1-\alpha)\left(\frac{LM}{|\alpha|}\right)^{\overline{\alpha}}x^{\frac{1}{1-\alpha}}.$
\end{lemma}

\section{Proof of Theorems \ref{S_1(H,M,N)} and \ref{S_1^{dag}(H,M,N)}}

\subsection{Proof of Theorem \ref{S_1(H,M,N)}}

\leavevmode \\

We first consider the case where $X > MN$. By applying Cauchy’s inequality, we obtain
\begin{equation*}
\begin{split}
 \left|S(H,M,N)\right|^2
    &\leqslant MN\sum_{m\sim M}\sum_{n\sim N}\left|\sum_{h\sim H}b(h)e\left(X\frac{(mn+\delta)^\beta h^\alpha}{M^\beta N^\beta H^\alpha}\right)\right|^2\\
    &\leqslant MN\sum_{m\sim M}\sum_{n\sim N}\sum_{h_1\sim H}\sum_{h_2\sim H}b(h_1)\overline{b(h_2)}e\left(X\frac{(mn+\delta)^\beta}{M^\beta N^\beta}\cdot\frac{h_1^{\alpha}-h_2^{\alpha}}{H^\alpha}\right)\\
    &\leqslant MN\sum_{m\sim M}\sum_{n\sim N}\sum_{h_1\sim H}\sum_{h_2\sim H}b(h_1)\overline{b(h_2)}e\bigg(Xu(m,n)v(h_1,h_2)\bigg)
\end{split}
\end{equation*}
where
$$
u(m,n)=\frac{(mn+\delta)^\beta}{M^\beta N^\beta}, \quad v(h_1, h_2)=\frac{h_1^\alpha-h_2^\alpha}{H^\alpha}.
$$
By Lemmas \ref{double large sieve inequality}, \ref{lemma:l1-4}, and \ref{lemma：mn}, we have
$$
 \left|S(H,M,N)\right|^2\ll MNX^{\frac{1}{2}}\left\{H^{4+\varepsilon}\bigg(\frac{1}{H^2}+\frac{1}{X}\bigg)\right\}^{\frac{1}{2}}\left\{(MN)^{2+\varepsilon}\left(\frac{1}{MN}+\frac{1}{X}\right)\right\}^{\frac{1}{2}}.
$$
Thus, 
\begin{equation}\label{S11}
\begin{split}
 S(H,M,N)
    &\ll(HMN)^{1+\varepsilon}\left\{\left(\frac{X}{MNH^2}\right)^{\frac{1}{4}}+\frac{1}{(MN)^{\frac{1}{4}}}+\frac{1}{H^{\frac{1}{2}}}+\frac{1}{X^{\frac{1}{4}}}\right\}\\
    &\ll(HMN)^{1+\varepsilon}\left\{\left(\frac{X}{MNH^2}\right)^{\frac{1}{4}}+\frac{1}{(MN)^{\frac{1}{4}}}+\frac{1}{H^{\frac{1}{2}}}\right\}.
\end{split}
\end{equation}

Next, we we consider the case $X \leqslant MN$. In this regime, we express the sum as
$$
S(H,M,N)=\sum_{m\sim M}\sum_{n\sim N}a(m,n)\sum_{h\sim H}b(h)e\bigg(Xu(m,n)v(h)\bigg),
$$
where
$$
u(m,n)=\frac{(mn+\delta)^\beta}{M^\beta N^\beta}, \quad v(h)=\frac{h^\alpha}{H^\alpha}.
$$
A direct application of Lemmas \ref{double large sieve inequality}, \ref{lemma:l1-2}, and \ref{lemma：mn} yields
\begin{equation}\label{S12}
\begin{split}
  S(H,M,N)
    &\ll X^{\frac{1}{2}}\left\{H^2\left(\frac{1}{H}+\frac{1}{X}\right)\right\}^{\frac{1}{2}}\left\{(MN)^{2+\varepsilon}\left(\frac{1}{MN}+\frac{1}{X}\right)\right\}^{\frac{1}{2}}\\
    &\ll (HMN)^{1+\varepsilon}\left\{\frac{1}{H^{\frac{1}{2}}}+\frac{1}{X^{\frac{1}{2}}}\right\}.
\end{split}
\end{equation}

Combining estimates \eqref{S11} and \eqref{S12}  completes the proof of Theorem \ref{S_1(H,M,N)}.

\subsection{Proof of Theorem \ref{S_1^{dag}(H,M,N)}}

\leavevmode \\

We divide the proof into three parts according to the range of $X$.

(1) When $X\ll H$, let $f(h)$ be defined as
$$
f(h)=X\frac{(mn+\delta)^{\beta}h^\alpha}{M^{\beta}N^{\beta}H^{\alpha}}.
$$
For $m \sim M$ and $n \sim N$, we observe that the derivative satisfies $|f'(h)| \leqslant 1/2$. 
Applying Lemma \ref{Kusmin-Landau}, we obtain the following bound for the inner sum
$$
\left|\sum_{h\sim H}e\left(X\frac{(mn+\delta)^{\beta}h^\alpha}{M^{\beta}N^{\beta}H^{\alpha}}\right)\right|^*\ll \frac{H}{X}.
$$
Summing over $m$ and $n$, it follows that
$$
S^*(H,M,N)\ll \frac{HMN}{X}.
$$

(2) When $X\gg H^2$, Let $T(m, n)$ denote the inner sum
$$
T(m,n):=\left|\sum_{h\sim H}e\left(X\frac{(mn+\delta)^{\beta}h^\alpha}{M^{\beta}N^{\beta}H^{\alpha}}\right)\right|^*.
$$
According to Lemma \ref{Z*} , the total sum can be estimated as
\begin{equation*}
\begin{split}
 S^*(H,M,N)
    = \sum_{m\sim M}\sum_{n\sim N}T(m,n)
    \leqslant \int _{\frac{1}{2}}^{\frac{1}{2}}\left(\sum_{m\sim M}\sum_{n\sim N}\left|\sum_{h\sim H}e\left(X\frac{(mn+\delta)^{\beta}h^\alpha}{M^{\beta}N^{\beta}H^{\alpha}}\right)e(ht)\right|\right)\mathscr{L}(t)\mathrm{d}t.
\end{split}
\end{equation*}
The three-dimensional exponential sum in  brackets is of type \eqref{type1} and an application of Theorem \ref{S_1(H,M,N)} shows that it is
\begin{equation*}
\begin{split}
 S^*(H,M,N)
    &\ll (HMN)^{1+\varepsilon}\left\{\left(\frac{X}{MNH^2}\right)^{\frac{1}{4}}+\frac{1}{(MN)^{\frac{1}{4}}}+\frac{1}{H^{\frac{1}{2}}}+\frac{1}{X^{\frac{1}{2}}}\right\}\\
    &\ll (HMN)^{1+\varepsilon}\left\{\left(\frac{X}{MNH^2}\right)^{\frac{1}{4}}+\frac{1}{H^{\frac{1}{2}}}\right\}.
\end{split}
\end{equation*}

(3) when $H\ll X\ll H^2$, we set
$$
x(m,n)=X\frac{(mn+\delta)^\beta}{M^\beta N^\beta}\asymp X.
$$
The sum can then be written as
$$
    S^*(H,M,N)=\sum_{m\sim M}\sum_{n\sim N}\left|\sum_{h\sim H}e\Bigg(x(m,n)\bigg(\frac{h}{H}\bigg)^{\alpha}\Bigg)\right|^*.
$$
According to the Lemma \ref{dependence}, there exist two positive integers $L$ and $L_1$ which do not depend on $m$ and $n$, a positive real number $Y$ with 
$$
Y\asymp L_1\asymp \frac{Y}{H}, \quad 1\leqslant L\leqslant L_1, \quad Y\asymp X,
$$
and such that
$$
 S^*(H,M,N)\ll \frac{H}{X^{\frac{1}{2}}}\sum_{m\sim M}\sum_{n\sim N}\left|e\left(Y\frac{(mn+\delta)^{\tilde{\beta}}l^{\overline{\alpha}}}{M^{\tilde{\beta}}N^{\tilde{\beta}}L^{\overline{\alpha}}}\right)\right|^*+MNH^{\frac{1}{2}},
$$
where
$$
\tilde{\beta}=\frac{\beta}{1-\alpha}, \quad \overline{\alpha}=\frac{\alpha}{\alpha-1}.
$$
By Lemma \ref{Z*} and Theorem \ref{S_1(H,M,N)}, we get
\begin{equation*}
\begin{split}
 S^*(H,M,N)
    &\ll \frac{H}{X^{\frac{1}{2}}}\int_{\frac{1}{2}}^{\frac{1}{2}}\left(\sum_{m\sim M}\sum_{n\sim N}\bigg|\sum_{L\leqslant l\leqslant L_1}e\bigg(Y\frac{(mn+\delta)^{\tilde{\beta}}l^{\overline{\alpha}}}{M^{\tilde{\beta}}N^{\tilde{\beta}}L^{\overline{\alpha}}}e(lt)\bigg)\bigg|\right)\mathscr{L}(t)\mathrm{d}t+MNH^{\frac{1}{2}}\\
     &\ll \frac{H}{X^{\frac{1}{2}}}(LMN)^{1+\varepsilon}\left\{\left(\frac{X}{MNL^2}\right)^{\frac{1}{4}}+\frac{1}{(MN)^{\frac{1}{4}}}+\frac{1}{L^{\frac{1}{2}}}\right\}+MNH^{\frac{1}{2}}\\
     &\ll (HMN)^{1+\varepsilon}\left\{\left(\frac{X}{MNH^2}\right)^{\frac{1}{4}}+\left(\frac{X^2}{MNH^4}\right)^{\frac{1}{4}}+\frac{1}{H^{\frac{1}{2}}}\right\}\\
      &\ll (HMN)^{1+\varepsilon}\left\{\left(\frac{X}{MNH^2}\right)^{\frac{1}{4}}+\frac{1}{H^{\frac{1}{2}}}\right\}.
\end{split}
\end{equation*}
This completes the proof of Theorem 
\ref{S_1^{dag}(H,M,N)}.

\section{Proof of Theorem \ref{Main Theorem}}
\subsection{Preliminary lemmas }

\leavevmode \\

In this subsection, we present several lemmas essential for the proof of Theorem \ref{Main Theorem}. We begin by establishing an explicit expression for the error term $\Delta(x)$ introduced in \eqref{sum tau(x)}.

\begin{lemma}\cite[Theorem 4.5]{GK}\label{delta(x)}
For any $x\in \mathbb{R}$, we have
$$
\Delta(x)=-2\sum_{n\leqslant \sqrt{x}}\psi\left(\frac{x}{n}\right)+O(1),
$$
where $\psi =x-[x]-1/2$.
\end{lemma}

The second lemma provides a trigonometric approximation for the function $\psi(x)$.
\begin{lemma}\cite[Lemma 4.1]{BDHPS}\label{psi}    
    For $x\geqslant 1$ and $H\geqslant1$, we have
    \begin{eqnarray*}
        \psi(x)=-\sum_{1\leqslant |h| \leqslant H}\Phi\left(\frac{h}{H+1}\right)\frac{e(hx)}{2\pi ih}+R_H(x),
    \end{eqnarray*}
    where $e(t):=e^{2\pi it}$, $\Phi(t):=\pi t(1-|t|)\cot(\pi t)+|t|$ and the error term $R_H(x)$ satisfies 
    \begin{eqnarray*}
        |R_H(x)|\leqslant \frac{1}{2H+2}\sum_{0\leqslant |h|\leqslant H}\left(1-\frac{|h|}{H+1}\right)e(hx).
    \end{eqnarray*}
\end{lemma}

The subsequent lemma offers an optimized version of van der Corput’s B-process (stationary phase method for exponential sums).

\begin{lemma}\cite[Theorem 1.2]{V}\label{lem:improved B process} 
Suppose that $f(x)$ and $g(x)$ are functions of real value with $f$ continuously differentiable four times and $g$ twice continuously differentiable on the interval $[a,b]$. Suppose that there are positive constants $M$, $T$, and $U$, with $M\asymp b-a$, such that, for $x\in [a,b]$,

   \[f^{''}(x)\geqslant T/M^2, \quad |f^{(2+r)}(x)|\ll T/M^{2+r}, \quad |g^{(r)}(x)|\ll U/M^{r},\]
for $r=0,1,2$. Then
\begin{equation*}
\begin{split}
\sideset{}{^*} \sum_{a \leqslant n \leqslant b} g(n)e(f(n))
=&\sideset{}{^*} \sum_{f^{'}(a)\leqslant r\leqslant f^{'}(b) } \frac{g(x_r)e(f(x_r)-rx_r+1/8)}{\sqrt{f^{''}(x_r)}}+O\bigg(U(R(a)+R(b))\bigg)\\
&\;\;\;+O\Bigg(U\bigg(\log{(f^{'}(b)-f^{'}(a)+2)}+\frac{M}{T}+\frac{T}{M^2}+1\bigg)\Bigg),
\end{split}
\end{equation*}
where $x_r$ is the unique solution to $f^{'}(x_r)=r$ in the interval $[a,b]$ and
\[R(\mu)=
\begin{cases}
0, & \|f^{'}(\mu)\|=0,\\
\min\left\{\frac{M}{\sqrt{T}}, \frac{1}{\|f^{'}(\mu)\|}\right\},&\|f^{'}(\mu)\|\neq0.
\end{cases}
\]
\end{lemma}

The following lemma is also crucial to the proof of  Theorem \ref{Main Theorem}.

\begin{lemma}\cite[Lemma 3]{J}\label{lem:min}
    Assume that when $x\sim N$, $f(x)\ll P$, $f^{'}(x)\gg \Delta$, then
\[\sum_{n\sim N}\min{\left\{Q, \frac{1}{\|f(n)\|}\right\}}\ll(P+1)\left(Q+\frac{1}{\Delta}\right)\log{\left(2+\frac{1}{\Delta}\right)}.\]
\end{lemma}

Our final lemma is derived from a result by Robert and Sargos concerning three-dimensional exponential sums.
\begin{lemma}\cite[Theorem 3]{RS}\label{three-dimensional exponential sums}
   Let $$
   S=\sum_{h\sim H}\sum_{n\sim N}\left|\sum_{m\sim M}e\left(X\frac{m^{\alpha}h^{\beta}n^{\gamma}}{M^{\alpha}H^{\beta}N^{\gamma}}\right)\right|^*,
   $$
   where $H$, $M$, $N$ are positive integers, $X$ is a real number greater than one, $\alpha, \beta, \gamma$ are fixed real  numbers such that $\alpha(\alpha-1)\beta\gamma \neq 0$, then we have
   $$
   S\ll (HMN)^{1+\varepsilon}\left\{\left(\frac{X}{HNM^2}\right)^{\frac{1}{4}}+\frac{1}{M^{\frac{1}{2}}}+\frac{1}{X}\right\}.
   $$
\end{lemma}

\subsection{A key inequality}

\leavevmode \\

For $x > 1$ and positive integers $H, D_1, D_2, L$, we define the following exponential sum with a
constant perturbation $\delta$: 
$$
\mathfrak{S}(x,D,H):=\sum_{h\sim H}\frac{1}{h}\sum_{d_1\sim D_1}\sum_{d_2\sim D_2}\sum_{l\sim L}e\left(\frac{hx}{l(d_1d_2+\delta)}\right)
$$
where   $D_1$, $D_2$ satisfy the condition $D_1 D_2 \sim D$. In this subsection, we establish an upper bound for $\mathfrak{S}(x, D, H)$, which serves as a crucial estimate in the proof of Theorem \ref{Main Theorem}.

\begin{proposition}\label{proposition}
    Let $\delta\in\{0,1\}$, $1\leqslant D<x$, and $L^2D\leqslant x$. For any $\varepsilon > 0$, we have the following upper bound
\begin{equation*}
\begin{split}
\mathfrak{S}(x,D,H)\ll
        &\bigg(x^{\frac{11}{30}}H^{\frac{1}{4}}D^{\frac{3}{8}}+x^{\frac{7}{30}}D^{\frac{3}{4}}+x^{\frac{23}{60}}H^{\frac{1}{2}}D^{\frac{1}{4}}
        +x^{\frac{4}{15}}H^{\frac{1}{2}}D^{\frac{1}{2}}\\
        & \hskip 4mm+x^{-\frac{1}{4}}H^{-\frac{1}{2}}D^{\frac{5}{4}}+D
        +x^{-\frac{2}{5}}HD^{\frac{3}{2}}\bigg)x^\varepsilon.
\end{split}
\end{equation*}
\end{proposition}

\begin{proof}

    We divide the proof of Proposition \ref{proposition} into two cases.

{\bf{Case I: $L^2D\leqslant x^{14/15}$.} }

We write
$$
\frac{hx}{l(d_1d_2+\delta)}=\frac{hx}{LD}\cdot \frac{l^{-1}(d_1d_2+\delta)^{-1}}{L^{-1}D_1^{-1}D_2^{-1}}.
$$
 Applying Theorem \ref{S_1^{dag}(H,M,N)} with  
 $$
 (X,M,N,H,\beta, \alpha)=(hx/(LD), D_1,D_2,L,-1,-1).
 $$
We get 
\begin{equation}\label{S1}
\begin{split}
\mathfrak{S}(x,D,H)
    &\ll\sum_{{h\sim H}}\frac{1}{h}\left\{\left(\frac{hx}{LD}/(DL^2)\right)^{\frac{1}{4}}+L^{-\frac{1}{2}}+\frac{DL}{hx}\right\} (DL)^{1+\varepsilon}\\ 
     &\ll\left\{x^{\frac{1}{4}}H^{\frac{1}{4}}D^{\frac{1}{2}}L^{\frac{1}{4}}+DL^{\frac{1}{2}}+\frac{D^2L^2}{Hx}\right\} x^{\varepsilon}\\ 
     &\ll\left\{x^{\frac{11}{30}}H^{\frac{1}{4}}D^{\frac{3}{8}}+x^{\frac{7}{30}}D^{\frac{3}{4}}+x^{-\frac{1}{15}}H^{-1}D\right\} x^{\varepsilon}.
\end{split}
\end{equation}
The final inequality follows from the assumption $L^2D \leqslant x^{14/15}$.

{\bf{Case II: $x^{14/15}<L^2D\leqslant x$.} }

Applying the Lemma \ref{lem:improved B process} to the sum over $l$, and combining the variables $d_1$ and $d_2$ in error terms, we have
\begin{eqnarray*}
    &&\mathfrak{S}(x,D,H)\\
    &&\ll\Bigg|\sum_{h\sim H}\frac{1}{h}\sum_{d_1\sim D_1}\sum_{d_2\sim D_2}\sum_{\frac{hx}{4d_1d_2L^2}<r\leqslant\frac{hx}{d_1d_2L^2}}e\left(2\sqrt{\frac{hxr}{d_1d_2+\delta}}\right)\cdot \left(\frac{hx}{(d_1d_2+\delta)r^3}\right)^{\frac{1}{4}}\Bigg|\\
    &&+\sum_{h\sim H}\frac{1}{h}\sum_{d\sim D} R(L,d,h)\\
    &&+ \sum_{h\sim H}\frac{1}{h}\sum_{d\sim D}\Bigg(\log{\bigg(\frac{hx}{4(d+\delta)L^2}-\frac{hx}{(d+\delta)L^2}+2\bigg)}+\frac{(d+\delta)L^2}{hx}+\frac{hx}{(d+\delta)L^3}+1\Bigg),
\end{eqnarray*}
where
$$
R(L,d,h)=\min \left\{\left(\frac{hx}{d+\delta}\right)^{-\frac{1}{2}}L^{\frac{3}{2}},\; \frac{1}{\|\frac{hx}{(d+\delta)L^2}\|}\right\}.
$$
For the error term on the second line, Lemma \ref{lem:min} then gives
\begin{equation*}
\begin{split}
\sum_{h\sim H}\frac{1}{h}\sum_{d\sim D} R(L,d,h)
    &\ll\sum_{h\sim H}\frac{1}{h}\left\{x^{\frac{1}{2}}h^{\frac{1}{2}}D^{-\frac{1}{2}}L^{-\frac{1}{2}}+D+x^{-\frac{1}{2}}h^{-\frac{1}{2}}D^{\frac{1}{2}}L^{\frac{3}{2}}+\frac{D^2L^2}{hx}\right\}x^{\varepsilon}\\
    &\ll\bigg(x^{\frac{1}{2}}H^{\frac{1}{2}}D^{-\frac{1}{2}}L^{-\frac{1}{2}}+D+x^{-\frac{1}{2}}H^{-\frac{1}{2}}D^{\frac{1}{2}}L^{\frac{3}{2}}\bigg)x^\varepsilon.
\end{split}
\end{equation*}
The last inequality holds, provided that $L^2D \leqslant x$.
The contribution of the error term on the last line above
$$O\left(Dx^{\varepsilon}+\frac{Hx^{1+\varepsilon}}{L^3}\right).$$
Combining the estimates above, we have
\begin{equation*}
\begin{split}
 \mathfrak{S}(x,D,H)
    \ll&\;\Bigg|\sum_{h\sim H}\frac{1}{h}\sum_{d_1\sim D_1}\sum_{d_2\sim D_2}\sum_{\frac{hx}{4d_1d_2L^2}<r\leqslant\frac{hx}{d_1d_2L^2}}e\left(2\sqrt{\frac{hxr}{d_1d_2+\delta}}\right)\cdot \left(\frac{hx}{(d_1d_2+\delta)r^3}\right)^{\frac{1}{4}}\Bigg|\\
    &+\bigg(\frac{Hx}{L^3}+x^{\frac{1}{2}}H^{\frac{1}{2}}D^{-\frac{1}{2}}L^{-\frac{1}{2}}+D+x^{-\frac{1}{2}}H^{-\frac{1}{2}}D^{\frac{1}{2}}L^{\frac{3}{2}}\bigg)x^\varepsilon.
\end{split}
\end{equation*}
Interchanging the orders of summation and applying the triangle inequality, the quadruple sum in absolute values is 
$$
\ll x^{-\frac{1}{2}}H^{-\frac{3}{2}}D^{\frac{3}{4}}L^{\frac{3}{2}}\sum_{h\sim H}\sum_{\frac{hx}{16DL^2}<r\leqslant \frac{hx}{DL^2}}\left|\sum_{d_2\sim D_2}\sum_{d_1\sim D_1}e\left(2\sqrt{\frac{hxr}{d_1d_2+\delta}}\right)\cdot \left(\frac{1}{d_1d_2+\delta}\right)^{\frac{1}{4}}\right|
$$
At this point, we may apply the triangle inequality once more and then apply a partial summation to the variable $d_1$. Note that by symmetry, we may apply this and the following argument to the variable $d_2$ as well. We have
\begin{equation}\label{mathfrac{S}}
\begin{split}
\mathfrak{S}(x,D,H)
    \ll&\; x^{-\frac{1}{2}}H^{-\frac{3}{2}}D^{\frac{1}{2}}L^{\frac{3}{2}}\sum_{h\sim H}\sum_{\frac{hx}{16DL^2}<r\leqslant \frac{hx}{DL^2}}\sum_{d_2\sim D_2}\left|\sum_{d_1\sim D_1}e\left(2\sqrt{\frac{hxr}{d_1d_2+\delta}}\right)\right|^*\\
    &+\bigg(\frac{Hx}{L^3}+x^{\frac{1}{2}}H^{\frac{1}{2}}D^{-\frac{1}{2}}L^{-\frac{1}{2}}+D+x^{-\frac{1}{2}}H^{-\frac{1}{2}}D^{\frac{1}{2}}L^{\frac{3}{2}}\bigg)x^\varepsilon.
\end{split}
\end{equation}
Set 
$$
\mathfrak{S}^*(H,D_1,D_2,L):=\sum_{h\sim H}\sum_{\frac{hx}{16DL^2}<r\leqslant \frac{hx}{DL^2}}\sum_{d_2\sim D_2}\left|\sum_{d_1\sim D_1}e\left(2\sqrt{\frac{hxr}{d_1d_2+\delta}}\right)\right|^*.
$$
Note that
$$
\sqrt{\frac{hxr}{d_1d_2+\delta}}=\frac{xh}{DL}\cdot \frac{r^{\frac{1}{2}}d_2^{-\frac{1}{2}}(d_1+\frac{\delta}{d_2})^{-\frac{1}{2}}}{\left(\frac{hx}{DL^2}\right)^{\frac{1}{2}}D_2^{-\frac{1}{2}}D_1^{-\frac{1}{2}}}.
$$
Setting $d'_1 = d_1 + \frac{\delta}{d_2}$ with $\delta \in \{0, 1\}$, we have $d'_1 \sim D_1$. Thus, we may write
$$
\mathfrak{S}^*(H,D_1,D_2,L)=\sum_{h\sim H}\sum_{\frac{hx}{16DL^2}<r\leqslant \frac{hx}{DL^2}}\sum_{d_2\sim D_2}\left|\sum_{d_1^{'}\sim D_1}e\left(2\frac{xh}{DL}\cdot \frac{r^{\frac{1}{2}}d_2^{-\frac{1}{2}}d_1^{'-\frac{1}{2}}}{\left(\frac{hx}{DL^2}\right)^{\frac{1}{2}}D_2^{-\frac{1}{2}}D_1^{-\frac{1}{2}}}\right)\right|^*.
$$
We now apply Lemma \ref{three-dimensional exponential sums} with $(X,H,M,N)=(2xh/(DL), xh/(DL^2), D_1, D_2)$
to the triple sum over $r$, $d_1^{'}$, $d_2$ to see that
\begin{eqnarray*}
   \mathfrak{S}^*(H,D_1,D_2,L)\ll \sum_{h\sim H}\frac{hx}{L^2}\left(L^{\frac{1}{4}}D^{-\frac{1}{4}}D_1^{-\frac{1}{4}}+D_1^{-\frac{1}{2}}+\frac{DL}{xh}\right)x^{\varepsilon}.
\end{eqnarray*}
Thus, the first line of \eqref{mathfrac{S}} is 
\begin{eqnarray*}
&&\ll    x^{-\frac{1}{2}}H^{-\frac{3}{2}}D^{\frac{1}{2}}L^{\frac{3}{2}}\sum_{h\sim H}\frac{hx}{L^2}\left(L^{\frac{1}{4}}D^{-\frac{1}{4}}D_1^{-\frac{1}{4}}+D_1^{-\frac{1}{2}}+\frac{DL}{xh}\right)x^{\varepsilon}\\
&&\ll \bigg(x^{\frac{1}{2}}H^{\frac{1}{2}}L^{-\frac{1}{4}}D_2^{\frac{1}{4}}+x^{\frac{1}{2}}H^{\frac{1}{2}}L^{-\frac{1}{2}}D_2^{\frac{1}{2}}+x^{-\frac{1}{2}}H^{-\frac{1}{2}}L^{\frac{1}{2}}D^{\frac{3}{2}}\bigg)x^\varepsilon.
\end{eqnarray*}
The final inequality follows from the fact that $D_1 D_2 \sim D$.
As we noted above, by the symmetry of the variables $d_1$ and $d_2$, the first line of \eqref{mathfrac{S}} is also
$$
\ll \bigg( x^{\frac{1}{2}}H^{\frac{1}{2}}L^{-\frac{1}{4}}D_1^{\frac{1}{4}}+x^{\frac{1}{2}}H^{\frac{1}{2}}L^{-\frac{1}{2}}D_1^{\frac{1}{2}}+x^{-\frac{1}{2}}H^{-\frac{1}{2}}L^{\frac{1}{2}}D^{\frac{3}{2}}\bigg)x^{\varepsilon}.
$$
Since $\min\{D_1,D_2\}\ll D^{\frac{1}{2}}$, we conclude that
\begin{equation}\label{S2}
\begin{split}
 \mathfrak{S}(x,D,H)
    \ll&  \bigg(x^{\frac{1}{2}}H^{\frac{1}{2}}L^{-\frac{1}{4}}D^{\frac{1}{8}}+x^{\frac{1}{2}}H^{\frac{1}{2}}L^{-\frac{1}{2}}D^{\frac{1}{4}}+x^{-\frac{1}{2}}H^{-\frac{1}{2}}L^{\frac{1}{2}}D^{\frac{3}{2}}\\ 
    &+\frac{Hx}{L^3}+x^{\frac{1}{2}}H^{\frac{1}{2}}D^{-\frac{1}{2}}L^{-\frac{1}{2}}+D+x^{-\frac{1}{2}}H^{-\frac{1}{2}}D^{\frac{1}{2}}L^{\frac{3}{2}}\bigg)x^\varepsilon\\ 
    \ll& \bigg(x^{\frac{23}{60}}H^{\frac{1}{2}}D^{\frac{1}{4}}
    +x^{\frac{4}{15}}H^{\frac{1}{2}}D^{\frac{1}{2}}
    +x^{-\frac{1}{4}}H^{-\frac{1}{2}}D^{\frac{5}{4}}+ D
    +x^{-\frac{2}{5}}HD^{\frac{3}{2}}\bigg)x^{\varepsilon}
\end{split}
\end{equation}
provided $x^{14/15}<L^2D\leqslant x$.
Together with \eqref{S1} and \eqref{S2}, this completes the proof of Proposition \ref{proposition}.
 
\end{proof}

\subsection{Proof of Theorem \ref{Main Theorem}}

\leavevmode \\

Let $N$ be a parameter that will be chosen later. Write
 $$
 T(x)=\sum_{n\leqslant x}\tau\left(\left[\frac{x}{n}\right]\right)\tau(n)=T_1(x)+T_2(x),
 $$
where 
 $$
 T_1(x):=\sum_{n\leqslant N}\tau\left(\left[\frac{x}{n}\right]\right)\tau(n),\quad T_2(x):=\sum_{N<n\leqslant  x}\tau\left(\left[\frac{x}{n}\right]\right)\tau(n).
 $$
Since $\tau(n)\ll n^{\varepsilon}$, we have
\begin{eqnarray}\label{T_1}
     T_1(x)=\sum_{n\leqslant N}\tau\left(\left[\frac{x}{n}\right]\right)\tau(n)\ll Nx^{\varepsilon}.
\end{eqnarray}
We have the familiar equivalence
 $$
 d=\left[\frac{x}{n}\right]\quad \Longleftrightarrow\quad \frac{x}{d+1}<n \leqslant \frac{x}{d},
 $$
and so
 $$
 T_2(x)=\sum_{d\leqslant \frac{x}{N}}\tau(d)\sum_{\frac{x}{d+1}<n\leqslant\frac{x}{d+1}}\tau(n).
 $$
Let $\Delta(x)$ be the term in Dirichlet's divisor problem, i.e.
 $$
 \sum_{n\leqslant x}\tau(x)=x\log{x}+(2\gamma-1)x+\Delta(x). 
 $$
Then
 $$
 T_2(x)=T_{21}(x)-T_{22}(x)+T_{\Delta}(x),
 $$
where
\begin{eqnarray*}
    &&T_{21}(x):=x(\log{x}+2\gamma-1)\sum_{d\leqslant\frac{x}{N}}\frac{\tau(d)}{d(d+1)},\\
    &&T_{22}(x):=x\sum_{d\leqslant \frac{x}{N}}\tau(d)\left(\frac{\log{d}}{d}-\frac{\log{(d+1)}}{d+1}\right),\\
    &&T_{\Delta}(x):=\sum_{d\leqslant \frac{x}{N}}\tau(d)\left(\Delta\left(\frac{x}{d}\right)-\Delta\bigg(\frac{x}{d+1}\bigg)\right).
\end{eqnarray*}
A straightforward calculation shows that
 $$
 \sum_{d>\frac{x}{N}}\frac{\tau(d)}{d(d+1)}\ll Nx^{\varepsilon-1}, \qquad \sum_{d>\frac{x}{N}}\tau(d)\left(\frac{\log{d}}{d}-\frac{\log{(d+1)}}{d+1}\right)\ll Nx^{\varepsilon-1},
 $$
and therefore
$$
T_2(x)=C_1x(\log{x}+2\gamma-1)-C_3x+T_{\Delta}(x)+O(Nx^{\varepsilon})
$$
where
$$
C_1=\sum_{d=1}^{\infty}\frac{\tau(d)}{d(d+1)}, \quad C_3=\sum_{d=1}^{\infty}\tau(d)\left(\frac{\log{d}}{d}-\frac{\log{(d+1)}}{d+1}\right).
$$
Combining this with \eqref{T_1}, we obtain
\begin{eqnarray}\label{TT}
    T(x)=C_1x\log{x}+C_2x+T_{\Delta}(x)+O(Nx^{\varepsilon})
\end{eqnarray}
with $C_2=C_1(2\gamma-1)-C_3$.

For $T_\Delta(x)$, according to Lemma \ref{delta(x)} and the definition of the divisor function, we obtain
\begin{eqnarray}\label{T_{delta}ll}
    T_{\Delta}(x)
    \ll x^{\varepsilon}\max_{\substack{D\leqslant \frac{x}{N}\\ L^2D\ll x}}\left|\sum_{d_1\sim D_1}\sum_{d_2\sim D_2}\sum_{l\sim L}\psi\left(\frac{x}{l(d_1d_2+\delta)}\right)\right|+O\left(\frac{x^{1+\varepsilon}}{N}\right)
\end{eqnarray}
where $\delta=\{0,1\}$, $D_1D_2\sim D$.
To estimate $T_\Delta(x)$, we consider the following sum
$$
S_\delta(x,N)=\sum_{d_1\sim D_1}\sum_{d_2\sim D_2}\sum_{l\sim L}\psi\left(\frac{x}{l(d_1d_2+\delta)}\right).
$$
Now applying Lemma \ref{psi}, we get
\begin{equation*}
\begin{split}
S_{\delta}(x,N)
    =&-\frac{1}{2\pi i}\sum_{1\leqslant|h|\leqslant H}\Phi\left(\frac{h}{H+1}\right)\frac{1}{h}\sum_{d_1\sim D_1}\sum_{d_2\sim D_2}\sum_{l\sim L}e \left(\frac{hx}{l(d_1d_2+\delta)}\right)\\
    &+\sum_{d_1\sim D_1}\sum_{d_2\sim D_2}\sum_{l\sim L}R_H \left(\frac{x}{l(d_1d_2+\delta)}\right)\\
    \ll& x^{\varepsilon}\max _{1\leqslant H_0\leqslant H}\left|\sum_{h\sim H_0}\frac{1}{h}\sum_{d_1\sim D_1}\sum_{d_2\sim D_2}\sum_{l\sim L}e \left(\frac{hx}{l(d_1d_2+\delta)}\right)\right|+\frac{LD}{H}.
\end{split}
\end{equation*}
Applying Proposition \ref{proposition} to the sum within the absolute value above, and using condition $L^2D \ll x$, we have
\begin{equation*}
\begin{split}
S_{\delta}(x,N)
    &\ll  x^{\varepsilon}\max _{H_0\leqslant H}\bigg\{
    x^{\frac{11}{30}}H_0^{\frac{1}{4}}D^{\frac{3}{8}}
    +x^{\frac{7}{30}}D^{\frac{3}{4}}
    +x^{\frac{23}{60}}H_0^{\frac{1}{2}}D^{\frac{1}{4}}
        +x^{\frac{4}{15}}H_0^{\frac{1}{2}}D^{\frac{1}{2}}\\
        &\qquad \qquad\qquad+x^{-\frac{1}{4}}H_0^{-\frac{1}{2}}D^{\frac{5}{4}}+D
        +x^{-\frac{2}{5}}H_0D^{\frac{3}{2}}\bigg\}+\frac{LD}{H}\\
        &\ll x^{\varepsilon}\bigg(
        x^{\frac{11}{30}}H^{\frac{1}{4}}D^{\frac{3}{8}}
    +x^{\frac{7}{30}}D^{\frac{3}{4}}
    +x^{\frac{23}{60}}H^{\frac{1}{2}}D^{\frac{1}{4}}
        +x^{\frac{4}{15}}H^{\frac{1}{2}}D^{\frac{1}{2}}\\
         &\quad \quad\quad+x^{-\frac{1}{4}}D^{\frac{5}{4}}+D
        +x^{-\frac{2}{5}}HD^{\frac{3}{2}}\bigg)+\frac{x^{\frac{1}{2}}D^{\frac{1}{2}}}{H}.
\end{split}
\end{equation*}
Combining the above with \eqref{TT} and \eqref{T_{delta}ll}, we obtain
\begin{equation*}
\begin{split}
T(x)=&C_1x\log{x}+C_2x+O\Bigg(Nx^{\varepsilon}+x^{1+\varepsilon}N^{-1}+ x^{\varepsilon}\max_{D\leqslant \frac{x}{N}}\bigg\{
        x^{\frac{11}{30}}H^{\frac{1}{4}}D^{\frac{3}{8}}
    +x^{\frac{7}{30}}D^{\frac{3}{4}}\\
    &+x^{\frac{23}{60}}H^{\frac{1}{2}}D^{\frac{1}{4}}+x^{\frac{4}{15}}H^{\frac{1}{2}}D^{\frac{1}{2}}+ x^{-\frac{1}{4}}D^{\frac{5}{4}}+D
        +x^{-\frac{2}{5}}HD^{\frac{3}{2}}+\frac{x^{\frac{1}{2}}D^{\frac{1}{2}}}{H}\bigg\}\Bigg).
\end{split}
\end{equation*}
Taking the maximum over $D \leqslant \frac{x}{N}$, it follows that
\begin{equation*}
\begin{split}
T(x)=&
    C_1x\log{x}+C_2x
    +O\bigg(x^\varepsilon N+x^{1+\varepsilon}N^{-1}
    +x^{\frac{89}{120}+\varepsilon}H^{\frac{1}{4}}N^{-\frac{3}{8}}
    +x^{\frac{59}{60}+\varepsilon}N^{-\frac{3}{4}}\\
    &+x^{\frac{19}{30}+\varepsilon}H^{\frac{1}{2}}N^{-\frac{1}{4}}
    +x^{\frac{23}{30}+\varepsilon}H^{\frac{1}{2}}N^{-\frac{1}{2}}
    +x^{1+\varepsilon}N^{-\frac{5}{4}}
    +x^{\frac{11}{10}+\varepsilon}HN^{-\frac{3}{2}}
    +x^{1+\varepsilon}H^{-1}N^{-\frac{1}{2}}\bigg)\\
        =& C_1x\log{x}+C_2x+O\left(x^{\frac{17}{30}+\varepsilon}\right).
\end{split}
\end{equation*}
Optimizing this expression by choosing
$$
N=x^{\frac{17}{30}}, \quad H=x^{\frac{3}{20}},
$$
we complete the proof of Theorem \ref{Main Theorem}.

\leavevmode \\

\noindent \textbf{Acknowledgments.} 
The author expresses sincere gratitude to Professor Kui Liu and Professor Meselem Karras for their valuable suggestions and insightful comments, which greatly improved the quality of this manuscript. Special thanks are also due to Joshua Stucky for his helpful discussions and contributions during the early stages of this work.

\end{document}